\documentclass{article}
\usepackage[utf8]{inputenc}
\usepackage{amsmath}
\usepackage{amsthm}
\usepackage{amsfonts}
\usepackage{amssymb}
\usepackage{tikz-cd}

\newtheorem{definition}{Definition}
\newtheorem{theorem}{Theorem}
\newtheorem{lemma}[theorem]{Lemma}
\newtheorem{corollary}[theorem]{Corollary}
\newtheorem{proposition}[theorem]{Proposition}

\theoremstyle{remark}
\newtheorem*{remark}{Remark}

\newcommand{\prop}{\textup{prop}}

\title{A note on the minimal tensor product and the C*-envelope of operator systems}
\author{Ian Koot}
\date{February 2022}

\begin{document}

\maketitle

\begin{abstract}
\noindent In this article, we show that the $C^*$-envelope of the minimal tensor product of two operator systems is isomorphic to the minimal tensor product of their $C^*$-envelopes. We do this by identifying the \v Silov boundary ideal of the minimal tensor product of two operator systems. Finally, as an application of this result, we show that the propagation number of the minimal tensor product of operator systems is the maximum of the propagation numbers of the factors.
\end{abstract}

\section{Introduction}
It is a well-known result that every unital $C^*$-algebra is isomorphic to a unital, self-adjoint and norm-closed subalgebra of the space of bounded operators on some Hilbert space. Generalizing this, there is a structure called an \textit{operator system}, which are isomorphic to unital, closed and self-adjoint \textit{subspaces} of the bounded operators on some Hilbert space. These operator systems can be abstractly characterized as vector spaces with extra structure in the form of an involution, and a notion of when a matrix with entries in the operator system is positive.

Any unital, closed and self-adjoint subspace of a $C^*$-algebra is an operator system. Conversely, because the bounded operators on a Hilbert space form a $C^*$-algebra, every operator system is isomorphic to a unital, closed and self-adjoint subspace of a $C^*$-algebra. For any given operator system, we can therefore ask the question: in which $C^*$-algebras can we identify a subspace that is isomorphic to the given operator system? It turns out that all such $C^*$-algebras can be projected onto a specific $C^*$-algebra with the same property.

Specifically, for an operator system $E$ we can consider all maps $i_A: E \hookrightarrow A$, where $A$ is a $C^*$-algebra, such that $i_A$ is an isomorphism onto its image, and without loss of generality we can assume $A$ to be the smallest $C^*$-subalgebra containing $E$. We call such a map a \textit{$C^*$-extension}. There is a specific $C^*$-extension, called the \textit{$C^*$-envelope} of $E$, which is in some sense the `smallest' $C^*$-extension. What's more, one can prove that any $C^*$-algebra $A$ that arises in a $C^*$-extension contains an ideal $I$ called the \textit{\v Silov boundary ideal}, such that $A/I$ is isomorphic to the $C^*$-envelope.

There are multiple characterizations of the \v Silov boundary ideal. In this article these different characterizations are used to clarify the relation between the $C^*$-envelope and the minimal tensor product of operator systems and $C^*$-algebras. Specifically, we do this by proving the following theorem, which is the main result of this article:
\begin{theorem}
\label{thm:Main}
Let $E \subseteq B(H)$ and $F \subseteq B(K)$ be operator systems, and let $I \subseteq C^*(E)$ and $J \subseteq C^*(F)$ be their respective \v Silov boundary ideals. Then $\ker q_I \otimes_{\min} q_J$ is the \v Silov boundary ideal for $E \otimes_{\min} F \subseteq B(H \overline{\otimes} K)$. Consequently, this means that
\[ C^*_{env}(E \otimes_{\min} F) \cong C^*_{env}(E) \otimes_{\min} C^*_{env}(F) \]
\end{theorem}
\noindent To prove this, we first formally introduces the relevant concepts such as the $C^*$-envelope and the \v Silov ideal in
Section \ref{sec:preliminaries}. The proof of Theorem \ref{thm:Main} is located in Section \ref{sec:proof}, where we identify the \v Silov ideal of the tensor product $E \otimes_{\min} F$ of two operator systems $E$ and $F$, which enables us to prove the isomorphism.

Finally, in Section \ref{sec:propagationNumber} we apply this identity to a property called the \textit{propagation number} of an operator system, introduced by Connes and Van Suijlekom in \cite{ConnesVSuijlekom2020}. This is a number associated to an operator system which, roughly speaking, describes how close it is to being a $C^*$-algebra; it turns out to behave particularly nice with respect to the tensor product, as we will see. 

\section{Preliminaries}
\label{sec:preliminaries}

We use the following definition of operator systems:
\begin{definition}
An \textbf{operator system} $E$ is a unital linear subspace $E \subseteq B(H)$ that is closed in the norm-topology and for which $e^* \in E$ for all $e \in E$.
\end{definition}
For the definitions and discussions of related concepts such as operator spaces, completely bounded maps, completely isometric maps and completely positive maps, we refer for example to  \cite{PaulsenCompletelyBoundedMaps}. 

Some remarks on notation: for a $C^*$-algebra $A$ with (closed, two-sided) ideal $I \subseteq A$ we will denote the canonical quotient map $A \longrightarrow A/I$ by $q_I$. For a linear map $\phi: E \longrightarrow F$, we write the induced map on matrix spaces as
\[ \phi^{(n)}: M_n(E) \longrightarrow M_n(F), \,\,\, (e_{ij}) \mapsto (\phi(e_{ij})). \]
Also, for subsets of vector spaces $S \subseteq V$ and $T \subseteq W$ we let
\[ S \otimes T := \text{span} \{ s \otimes t \mid s \in S, t \in T\} \subseteq V \otimes W. \]
For linear maps $\phi: V_1 \longrightarrow V_2$ and $\psi: W_1 \longrightarrow W_2$ we let $\phi \otimes \psi: V_1 \otimes V_2 \longrightarrow W_1 \otimes W_2$ be the map defined by $\phi\otimes \psi (v \otimes w) = \phi(v)\otimes \psi(w)$ on elementary tensors. If $V_1, V_2, W_1$ and $W_2$ are more specifically operator spaces, operator systems or $C^*$-algebras, and $\phi$ and $\psi$ bounded, we will denote by $\phi \otimes_{\min} \psi$ the map obtained from $\phi \otimes \psi$ through continuous extension to the minimal tensor product (see Section \ref{sec:minimalTensorProduct}).

\subsection{The $C^*$-envelope and \v Silov ideal}
We start with the following definition due to Hamana \cite{Hamana1979}:
\begin{definition}
A \textbf{$C^*$-extension} of an operator system $E$ is a pair $(i,A)$ of a $C^*$-algebra $A$  and a linear map $i: E \longrightarrow A$ such that $C^*(i(E)) = A$, such that $i$ is a unital complete order isomorphism onto its image.
The \textbf{$C^*$-envelope} of an operator system $E$ is a $C^*$-extension $(i_E, C^*_{env}(E))$ with the following property:
\begin{itemize}
    \item for every $C^*$-extension $(i,A)$ there is a surjective *-homomorphism $\pi: A \longrightarrow C^*_{env}(E)$ such that
\[\begin{tikzcd}
                                               & A \arrow[d, "\pi", two heads] \\
E \arrow[ru, "i", hook] \arrow[r, "i_E"', hook] & C^*_{env}(E)               
\end{tikzcd}\]
commutes.
\end{itemize}
\end{definition}
Existence and uniqueness are proven by Hamana in \cite{Hamana1979}; note that there another definition is used, but by Corollary 4.2 in \cite{Hamana1979} and uniqueness the two definitions are equivalent. In particular we see that for any concrete operator system $E \subseteq B(H)$ we can construct the $C^*$-extension given by $C^*(E) \subseteq B(H)$, from which it follows that $C^*_{env}(E) \cong C^*(E)/\text{ker }\pi$. Hamana showed that the ideal $\ker \pi$ is in fact the so-called \textit{\v Silov boundary ideal}, as defined by Arveson in \cite{Arveson1969}. 

\begin{definition}
Let $E \subseteq B(H)$ be an operator system. A closed ideal $I \subseteq C^*(E)$ is a \textbf{boundary ideal} if the quotient map $C^*(E) \longrightarrow C^*(E)/I$ is completely isometric when restricted to $E$. The boundary ideal that is maximal with respect to inclusion is called the \textbf{\v Silov boundary ideal} (or simply \textbf{\v Silov ideal}).
\end{definition}

There is a more concrete description of the \v Silov boundary ideal, as first described by Arveson in \cite{Arveson1969}. It makes use of the following concept:

\begin{definition}
Let $E \subseteq B(H)$ be an operator system, and let $\sigma: C^*(E) \longrightarrow B(K)$ be an irreducible representation of $C^*(E)$. Then $\sigma$ is a \textbf{boundary representation} if for any unital completely positive map $\hat{\sigma}:C^*(E) \longrightarrow B(K)$ such that $\hat{\sigma}|_E = \sigma|_E$, we must necessarily have $\hat{\sigma} = \sigma$. We will denote the set of all boundary representations of $E$ by $\partial E$.
\end{definition}

Note that $\partial E$ really depends on the specific embedding $E \subseteq B(H)$ we choose, and not just on the abstract operator system structure.

\begin{proposition}
\label{prop:DavidsonKennedy}
Let $E \subseteq B(H)$ be an operator system and let $I \subseteq C^*(E)$ be the \v Silov ideal. Then
\[ I = \bigcap \{ \ker \sigma \mid \sigma \in \partial E \} \]
\end{proposition}

Proposition \ref{prop:DavidsonKennedy} had been proven in specific cases (see for example \cite{Arveson2008}), but was proven for the general case by Davidson and Kennedy in \cite{DavidsonKennedy2015}.

\subsection{The minimal tensor product}
\label{sec:minimalTensorProduct}
Finally, we introduce the minimal tensor product of operator systems. Given two Hilbert spaces $H$ and $K$, we can form the algebraic tensor product vector space $H \otimes K$. If we supply this vector space with the inner product given by
\[  \langle x \otimes y, v\otimes w \rangle_{H \otimes K} := \langle x,v\rangle_H \langle y,w \rangle_K\]
and complete the space with respect to the induced norm, we have created a new Hilbert space, which we will denote as $H \overline{\otimes} K$. If we now have $e \in B(H)$ and $f \in B(K)$, we can define an operator $e \otimes f \in B(H \overline{\otimes} K)$ through
\[ e\otimes f : \sum_i x_i\otimes y_i \mapsto \sum_i e(x_i)\otimes f(y_i) \]
and extending by continuity. Given two operator spaces $E \subseteq B(H)$ and $F \subseteq B(K)$, we can then define the linear subspace
\[ E \otimes_{\min} F := \overline{\text{span}} \{ e \otimes f \mid e \in E, f \in F\} \subseteq B(H \overline{\otimes} K), \]
where the closure is taken in the operator norm. If $E$ and $F$ are operator \textit{systems}, it is also easily seen to be self-adjoint and unital by the fact that $E$ and $F$ are self-adjoint and unital. So then $E \otimes_{\min} F$ is an operator system. In fact, if $A$ and $B$ are $C^*$-algebras (so in particular operator systems), then $A \otimes_{\min} B$ as defined above is a $C^*$-algebra. The reason this construction is called the \textit{minimal} tensor product is that the norm induced by $B(H \overline{\otimes} K)$ on the \textit{algebraic} tensor product $A \otimes B$ is the smallest $C^*$-norm on that space.

\section{Identifying the \v Silov ideal in the tensor product}
\label{sec:proof}
In this section, let $E \subseteq B(H)$ and $F \subseteq B(K)$ be operator systems, and let $I \subseteq C^*(E)$ and $J \subseteq C^*(F)$ be the \v Silov boundary ideals for $E$ and $F$, respectively.

Note that $C^*(E \otimes_{\min} F) = C^*(E) \otimes_{\min} C^*(F)$:  on the one hand $C^*(E) \otimes_{\min} C^*(F)$ is a $C^*$-algebra containing $E \otimes_{\min} F$, so the left-to-right inclusion holds; on the other hand, we can form the set
\[ \{ x \in B(H) \mid x \otimes 1 \in C^*(E \otimes_{\min} F) \} \]
which is clearly a $C^*$-algebra containing $E$, so it must contain $C^*(E)$, and we therefore have $C^*(E) \otimes 1 \subseteq  C^*(E \otimes_{\min} F)$, from which, after a similar argument for $F$, the right-to-left inclusion follows. In the next two subsections, we will show that the ideal
\[ \ker q_I \otimes_{\min} q_J \subseteq C^*(E \otimes_{\min} F)\]
is in fact the \v Silov boundary ideal for $E \otimes_{\min} F$. For this we first show that it is a boundary ideal, and after that we show that the \v Silov ideal is contained in $\ker q_I \otimes_{\min} q_J$, making use of the characterization through boundary representations. Since the \v Silov ideal is the maximal boundary ideal, we have that any boundary ideal that contains it must be the \v Silov ideal itself.

\subsection{Boundary ideal}
In order to prove that the qoutient by $\ker q_I \otimes_{\min} q_J$ is completely isometric on $E \otimes_{\min} F$, we prove that $q_I \otimes_{\min} q_J$ is completely isometric on $E \otimes_{\min} F$. This is sufficient because of the following lemma:

\begin{lemma}
\label{lem:equivalent_characterization}
Let $\pi: A \longrightarrow B$ be a *-homomorphism between unital $C^*$-algebras, and let $E \subseteq A$ be an operator sytem. Then $\pi|_{E}$ is completely isometric if and only if $q_{\ker \pi} |_E$ is completely isometric.
\end{lemma}
\begin{proof}
Note that we have an *-isomorphism $\phi: A/\ker \pi \longrightarrow \pi(B)$, with the property that $\phi \circ q_{\ker \pi} = \pi$. Since *-homomorphisms, the restriction of a completely isometric map to an operator system and the composition of complete isometric maps are all completely isometric, the lemma follows.
\end{proof}

What remains is to prove that the tensor product of two completely isometric maps is again completely isometric. For this, we use the argumentation from \cite[Sec. 1.5.1]{blecherMerdyOperatorAlgebras}, where the proof of Lemma \ref{lem:tensor_of_isometric_maps_is_isometric} is given as a sketch. Given operator \textit{spaces} $E$ and $F$, we can give the space of completely bounded maps between them, which we call $CB(E,F)$, an (abstract) operator space structure through the identification
\[ M_n(CB(E,F)) \cong CB(E,M_n(F)). \]
We then define the \textbf{dual} operator space as $E^*:= CB(E, \mathbb{C})$. We can use this operator space to describe the minimal tensor product as
\[ E \otimes_{\min} F \cong CB(F^*,E). \]
A discussion of these identifications can for example be found in \cite{EffrosRuanOperatorSpaces}, \cite{PisierIntroToOperatorSpaceTheory} and \cite{PaulsenCompletelyBoundedMaps}.

\begin{lemma}
\label{lem:completely_isometric_inclusion}
Let $\phi: E_1 \longrightarrow E_2$ be a completely isometric map. Then its pushforward $\phi_*: CB(F,E_1) \longrightarrow CB(F,E_2)$ given by $f \mapsto \phi \circ f$ is completely isometric.
\end{lemma}
\begin{proof}
Note that with the induced map 
\[ (\phi_*)^{(n)}: M_n(CB(F,E_1)) \longrightarrow M_n(CB(F,E_2)) \]
we can construct the commuting diagram
\[ \begin{tikzcd}
{M_n(CB(F,E_1))} \arrow[r, "(\phi_*)^{(n)}"] \arrow[d, "\sim"] & {M_n(CB(F,E_2))}                  \\
{CB(F,M_n(E_1))} \arrow[r, "(\phi^{(n)})_*"]                           & {CB(F,M_n(E_2))} \arrow[u, "\sim"]
\end{tikzcd} \]
We note that 
\[ \|(\phi^{(n)})_*(f)\|_{cb} = \|\phi^{(n)} \circ f \|_{cb} = \sup \{ \|(\phi^{(n)} \circ f)^{(m)}\|_m \mid 1 \leq m \} \]
and
\[ \|(\phi^{(n)} \circ f)^{(m)} \|_m = \sup \{ \|\phi^{(nm)} ( f^{(m)}(x))\|_m \mid x \in M_m(F), \|x\|_m \leq 1 \}. \]
Because $\phi$ is completely isometric, $\phi^{(nm)}$ is isometric, and so $\|(\phi^{(n)} \circ f)^{(m)}\|_m = \|f^{(m)}\|_m$. So $(\phi^{(n)})_*$ is isometric, and $(\phi_*)^{(n)}$ is then the composition of isometric maps, making it isometric. So $\phi_*$ is completely isometric.
\end{proof}

\begin{lemma}
\label{lem:tensor_of_isometric_maps_is_isometric}
If $\phi: E_1 \longrightarrow E_2$ and $\psi: F_1 \longrightarrow F_2$ are complete isometries between operator spaces, then $\phi\otimes_{\min} \psi: E_1 \otimes_{\min} F_1 \longrightarrow E_2 \otimes_{\min} F_2$ is a complete isometry. 
\end{lemma}
\begin{proof}
First, note that $\phi\otimes_{\min}\psi = \phi \otimes_{\min} I_{F_2} \circ I_{E_1}\otimes_{\min} \psi$, where $I_{E_1}$ and $I_{F_2}$ is the identity on $E_1$ and $F_2$, respectively. It therefore suffices to show that $\phi \otimes_{\min} I_{F_2}$ is a complete isometry. 

So let $\phi:E_1 \longrightarrow E_2$ be a complete isometry, and $F$ another operator space. Then because $\phi$ is completely isometric, by Lemma \ref{lem:completely_isometric_inclusion} it induces the completely isometric pushforward
\[ \phi_*: CB(F^*,E_1) \longrightarrow CB(F^*, E_2), f \mapsto \phi \circ f. \]
We therefore have the commuting diagram
\[
\begin{tikzcd}
{CB(F^*,E_1)} \arrow[r, "\phi_*"]                          & {CB(F^*,E_2)}                 \\
E_1\otimes_{\min} F \arrow[u, "i_1"] \arrow[r, "\phi\otimes I_F"] & E_2\otimes_{\min} F \arrow[u, "i_2"]
\end{tikzcd} \]
Since $\phi_*\circ i_1$ is completely isometric as the composition of completely isometric maps, we have that $i_2 \circ \phi\otimes I_F$ is also completely isometric. So $\phi\otimes I_F$ is completely isometric.
\end{proof}

\begin{corollary}
\label{cor:BoundaryIdeal}
Let $E \subseteq B(H)$ and $F \subseteq B(K)$ be operator systems, and $I$ and $J$ their respective \v Silov boundary ideals. Then $\ker q_I \otimes_{\min} q_J$ is a boundary ideal for $E \otimes_{\min} F \subseteq B(H \overline{\otimes} K)$.
\end{corollary}
\begin{proof}
By Lemma \ref{lem:equivalent_characterization} we need to prove that the map $q_I \otimes_{\min} q_J|_{E \otimes_{\min} F}  $ is completely isometric. Since $I$ and $J$ are by assumption boundary ideals, we have that $q_I|_E$ and $q_J|_F$ are completely isometric. So by Lemma \ref{lem:tensor_of_isometric_maps_is_isometric} we have that
\[ q_I \otimes_{\min} q_J|_{E \otimes_{\min} F} = q_I|_E \otimes_{\min} q_J|_F \]
is completely isometric.
\end{proof}

\subsection{Boundary representations}
We start this section with a crucial property of boundary representations. 
\begin{lemma}
\label{lem:Hopenwasser}
Let $E \subseteq B(H)$ and $F \subseteq B(K)$ be operator systems. If $\sigma_1 \in \partial E$ and $\sigma_2 \in \partial F$, then $\sigma_1 \otimes_{\min} \sigma_2 \in \partial (E \otimes_{\min} F)$.
\end{lemma}
The proof is a combination of \cite[Lemma 3]{Hopenwasser1978}, where it is proved that $\sigma_1\otimes_{\min} \sigma_2$ has a unique completely positive extension from $E \otimes_{\min} F$, and for example \cite[Proposition IV.4.13]{TakesakiTheory}, where it is proved that $\sigma_1 \otimes_{\min} \sigma_2$ is irreducible.

Next, we have the main lemma for this section. The proof is adapted from \cite[Lemma 2.2]{lazar2010}, with the primary difference being that the proof below references the Kirchberg Slice Lemma \cite[Lemma I.4.1.9]{RordamStormer} for the fact that an ideal in the minimal tensor product contains a simple tensor.

\begin{lemma}
\label{lem:lazar}
Let $A$ and $B$ be $C^*$-algebras, and let $\mathcal{K}$ and $\mathcal{L}$ be families of ideals in $A$ and $B$, respectively. Now define
\[ I = \bigcap_{K \in \mathcal{K}} K \,\,\, \textup{and} \,\,\, J = \bigcap_{L \in \mathcal{L}} L \]
Then
\[ \ker q_{I}\otimes_{\min} q_{J} = \bigcap\{ \ker q_{K} \otimes_{\min} q_{L} \mid (K, L) \in \mathcal{K} \times \mathcal{L} \} \]
\end{lemma}
\begin{proof}
For readability, we will write $\otimes_{\min}$ simply as $\otimes$, and we will write $M = \bigcap\{ \ker q_{K} \otimes q_{L} \mid (K , L) \in \mathcal{K} \times \mathcal{L} \} $.
\begin{itemize}

    \item First, we prove the inclusion $ \ker q_{I}\otimes q_{J} \subseteq M $.
For all $K \in \mathcal{K}$ and $L \in \mathcal{L}$, we have that $I \subseteq K$ and $J \subseteq L$. So we can apply the isomorphisms $\rho: (A/I)/(K/I) \xrightarrow{\sim} A/K$ and $\sigma: (B/J)/(L/J) \xrightarrow{\sim} B/L$ to see that
\[ q_{K}\otimes q_{L} = (\rho \otimes \sigma)(q_{K/I} \otimes q_{L/J})(q_{I}\otimes q_{J})\]
so $\ker q_{I} \otimes q_{J} \subseteq \ker q_{K}\otimes q_{L}$, proving the inclusion.

    \item Second, for the inclusion
$ \ker q_{I}\otimes q_{J} \supseteq M, $ we define the seminorm
    \[ N(x) = \sup \{ \| q_{K/I}\otimes q_{L/J}(x) \| \mid  (K,L)\in \mathcal{K} \times \mathcal{L} \}. \]
    on $A/I\otimes B/J$. It is easily seen that $N(xy) \leq N(x)N(y)$ and $N(x^*x) = N(x)^2$. Note that $N(x) = 0$ if and only if $x \in \ker q_{K/I}\otimes q_{L/J}$ for all $(K,L) \in \mathcal{K} \times \mathcal{L}$. In particular we have that $\ker N$ is the intersection of ideals, and so it is an ideal itself.
    
    Note that if $[x]\otimes[y] \in A/I \otimes B/J$ is nonzero, then there are $K$ and $L$ such that $x \not\in K$ and $y \not\in L$. But 
    \[ 0 \neq q_{K}\otimes q_{L}(x\otimes y) = (\rho \otimes\sigma )(q_{K/I}\otimes q_{L/J})(q_{I}\otimes q_{J})(x\otimes y) \] \[= (\rho \otimes\sigma)(q_{K/I} \otimes q_{L/J})([x]\otimes [y])\]
    so
    \[ N([x]\otimes[y]) \geq \|(q_{K/I}\otimes q_{L/J})([x]\otimes[y])\| = \|q_{K} \otimes q_{L} (x \otimes y) \| > 0. \]
    Specifically, $\ker N$ does not contain any simple tensors. But by the Kirchberg Slice Lemma \cite[Lemma I.4.1.9]{RordamStormer} every nontrivial ideal contains a nontrivial simple tensor. So this means that $\ker N$ is trivial, and so $N$ is actually a $C^*$-norm. But those are unique, so $N(x) = \|x\|$ on $A/I \otimes B/J$.
    
    Finally, as we remarked above, we have
    \[  q_{K}\otimes q_{L}(x) = (\rho \otimes \sigma)(q_{K/I} \otimes q_{L/J})(q_{I}\otimes q_{J}) (x). \]
    So if $x \in M$, then the left-hand side is zero for all $(K,L) \in \mathcal{K} \times \mathcal{L}$, and $(q_{I}\otimes q_{J})(x) \in \ker q_{K/I} \otimes q_{L/J}$ for all $(K,L) \in \mathcal{K} \times \mathcal{L}$ because $\rho \otimes \sigma$ is injective. We therefore have that 
    \[ \| q_{I}\otimes q_{J}(x) \| = N(q_{I}\otimes q_{J}(x)) = 0, \] so $x \in \ker q_{I}\otimes q_{J}$, so that indeed $M \subseteq \ker q_I\otimes q_J$.
\end{itemize}
In conclusion, we have that $M = \ker q_I \otimes q_J$, proving the lemma.
\end{proof}

\begin{corollary}
\label{cor:contains_Silov_ideal}
Let $E \subseteq B(H)$ and $F \subseteq B(K)$ be operator systems, and let $I$ and $J$ be their respective \v Silov boundary ideals. Then
\[ \ker q_I \otimes_{\min} q_J = \bigcap \{ \ker \sigma_1 \otimes_{\min} \sigma_2 \mid (\sigma_1, \sigma_2) \in \partial E \times \partial F \} \]
Specifically, this means that $\ker q_I \otimes_{min} q_J$ contains the \v Silov boundary ideal of $E \otimes_{\min} F$.
\end{corollary}
\begin{proof}
Using Proposition \ref{prop:DavidsonKennedy} we see that
\[ I = \bigcap_{\sigma_1 \in \partial E} \ker \sigma_1, \,\,\, J = \bigcap_{\sigma_2 \in \partial F}  \ker \sigma_2.\]
According to Lemma \ref{lem:lazar}, we therefore have that
\[ \ker q_I \otimes_{\min} q_J = \bigcap \{ \ker q_{\ker \sigma_1} \otimes_{\min} q_{\ker \sigma_2} \mid (\sigma_1, \sigma_2) \in \partial E \times \partial F \}. \]
However, because for $i=1,2$ there exists a *-isomorphism $\phi_i$ such that $\phi_i \circ \sigma_i = q_{\ker \sigma_i}$, we have that
\[ \ker q_{\ker \sigma_1} \otimes_{\min} q_{\ker \sigma_2} = \ker \sigma_1 \otimes_{\min} \sigma_2 \]
for all $\sigma_1 \in \partial E$ and $\sigma_2 \in \partial F$. So we can indeed conclude that
\[ \ker q_I \otimes_{\min} q_J = \bigcap \{ \ker \sigma_1 \otimes_{\min} \sigma_2 \mid (\sigma_1, \sigma_2) \in \partial E \times \partial F \}. \]
Moreover, since $(\sigma_1,\sigma_2) \in \partial E \times \partial F$ implies that $\sigma_1 \otimes_{\min} \sigma_2 \in \partial(E \otimes_{\min} F) $ by Lemma \ref{lem:Hopenwasser}, we have that
\[  \bigcap \{ \ker \sigma_1 \otimes_{\min} \sigma_2 \mid (\sigma_1, \sigma_2) \in \partial E \times \partial F \} \supseteq \bigcap \{ \ker \sigma \mid \sigma \in \partial (E \otimes_{\min} F)\}. \]
So indeed, $\ker q_I \otimes_{\min} q_J$ contains the \v Silov boundary ideal.
\end{proof}

We now arrive at the proof of Theorem \ref{thm:Main}, where we prove that $\ker q_I \otimes_{\min} q_J$ satisfies the properties of the \v Silov ideal. Since we can use the \v Silov ideal to construct the $C^*$-envelope, this immediately leads to a result about $C^*$-envelopes.

\begin{proof}[Proof of Theorem \ref{thm:Main}]
By Corollary \ref{cor:contains_Silov_ideal} we know that $\ker q_I \otimes_{\min} q_J$ contains the \v Silov boundary ideal for $E \otimes_{\min} F$. But by Corollary \ref{cor:BoundaryIdeal} we know that it is a boundary ideal itself. Since the \v Silov boundary ideal is the maximal boundary ideal, we must have that $\ker q_I \otimes_{\min} q_J$ is the \v Silov boundary ideal for $E \otimes_{\min} F \subseteq B(H\overline{\otimes} K)$.

In particular this means that
\[ C^*_{env}(E \otimes_{\min} F) \cong C^*(E \otimes_{\min} F)/\ker q_I \otimes_{\min} q_J \] \[ = C^*(E) \otimes_{\min} C^*(F) / \ker q_I \otimes_{\min} q_J.  \]
Because $q_I: C^*(E) \longrightarrow C^*(E)/I \cong C^*_{env}(E)$ and $q_J: C^*(F) \longrightarrow C^*(F)/I \cong C^*_{env}(F)$, both being surjective *-homomorphisms, we have that
\[ q_I \otimes_{\min} q_J : C^*(E)\otimes_{\min} C^*(F) \longrightarrow  C^*_{env}(E) \otimes_{\min} C^*_{env}(F) \]
is a surjective *-homomorphism, so
\[ C^*(E)\otimes_{\min} C^*(F)/ \ker q_I \otimes_{\min} q_J \cong C^*_{env}(E) \otimes_{\min} C^*_{env}(F), \]
proving the theorem.
\end{proof}

\section{Application: the propagation number}
\label{sec:propagationNumber}

In \cite{ConnesVSuijlekom2020}, Connes and Van Suijlekom define a property of operator systems called the propagation number. The definition of this property is based on the $C^*$-envelope. Using Theorem \ref{thm:Main} we can fairly directly determine the behaviour of this property under the tensor product.

\begin{definition}
For $E \subseteq C^*(E)$ a concrete operator system, let
\[ E^{\circ n} := \overline{\textup{span}}\{ e_1 \cdot \ldots \cdot e_n \mid e_i \in E \textup{  for  } i=1, \ldots,n\} \subseteq C^*(E).\]
Let $(i_E, C^*_{env}(E))$ be the $C^*$-envelope of $E$. Then the \textbf{propagation number} is defined as
\[ \prop(E):= \inf \{ n \in \mathbb{N} \mid (i_E(E))^{\circ n} = C^*_{env}(E)\}. \]
\end{definition}

\begin{proposition}
Let $E \subseteq A$ and $F \subseteq B$ be concrete operator systems. Then
\[ E^{\circ n}\otimes_{\min} F^{\circ n} = (E\otimes_{\min} F)^{\circ n} \subseteq A \otimes_{\min} B. \]
\end{proposition}
\begin{proof}
Because
\[ \left(\sum_i e_{1i}\otimes f_{1i} \right)\left( \sum_j e_{2j}\otimes f_{2j}\right) = \sum_{i,j} (e_{1i}e_{2j})\otimes (f_{1i}f_{2j})\]
we have that
\[ \{ x_1 x_2  \cdots  x_n \mid x_i \in E \otimes F\} \subseteq E^{\circ n}\otimes F^{\circ n} \subseteq E^{\circ n}\otimes_{\min} F^{\circ n}.\]
Each element in $(E \otimes_{\min} F)^{\circ n}$ is of the form $y_1y_2\cdots y_n$ for $y_i \in E \otimes_{\min} F$, while each $y_i$ is the limit of some sequence $(x_{ij})_{j=1}^\infty$ in $E \otimes F$. But by continuity of multiplication, this means that $x_{1j}x_{2j}\cdots x_{nj} \rightarrow y_1y_2\cdots y_n$ as $j \rightarrow \infty$ and so
\[ (E \otimes_{\min} F)^{\circ n} \subseteq \overline{\{ x_1 x_2  \cdots  x_n \mid x_i \in E \otimes F\}} \subseteq E^{\circ n} \otimes_{\min} F^{\circ n}.\]
Conversely, an element of the form $\sum_{i} (e_{i1}e_{i2}\ldots e_{in})\otimes(f_{i1}f_{i2}\ldots f_{in})$ can be seen to lie in the linear span of elements of the form $(e_{i1}\otimes f_{i1})(e_{2i}\otimes f_{2i})\ldots (e_{in}\otimes f_{in}) \in (E\otimes_{\min} F)^{\circ n}$. So $E^{\circ n} \otimes F^{\circ n} \subseteq (E\otimes_{\min} F)^{\circ n}$, and by the fact that $(E \otimes_{\min} F)^{\circ n}$ is closed, we have that $E^{\circ n} \otimes_{\min} F^{\circ n} \subseteq (E \otimes_{\min} F)^{\circ n}$.
\end{proof}

\begin{lemma}
\label{lem:separation}
Let $A$ and $B$ be $C^*$-algebras, and let $S \subsetneq A$ be a proper closed subspace. Then $S \otimes_{min} B$ is a proper subspace of $A \otimes_{\min} B$.
\end{lemma}
\begin{proof}
Take $x \in A\setminus S$, and let $f_x \in A^*$ with $f_x(x) = 1$ and $f_x(S) = \{0\}$ (such an element exists by a standard separation result, e.g. \cite[Corollary IV.3.15]{ConwayACourse}). Take $b \in B$ nonzero, and $f_b \in B^*$ such that $f_b(b) = 1$. Then define the continuous functional 
$$g:= f_x \otimes_{\min} f_b: A \otimes_{\min} B \longrightarrow \mathbb{C}\otimes \mathbb{C} \cong \mathbb{C}. $$
Clearly, $g(x \otimes b) = 1$, and $g(S \otimes B) = \{0\}$. So by continuity we have $g(S \otimes_{\min} B) = \{0\}$, and we must therefore have that $x \otimes b \not\in S \otimes_{\min} B$.
\end{proof}

\begin{theorem}
\label{thm:propagationNumber}
Let $E$ and $F$ be operator systems. Then
\[ \prop(E \otimes_{\min} F) = \max\{\prop(E),\prop(F)\}. \]
\end{theorem}
\begin{proof}
We identify $E \otimes_{\min} F$ with its image in its $C^*$-envelope, which by Theorem \ref{thm:Main} is $C^*_{env}(E) \otimes_{\min} C^*_{env}(F)$. Let $M = \max \{ \prop(E),\prop(F) \}$. Then
\[ (E \otimes_{\min} F)^{\circ M} = E^{\circ M} \otimes_{\min} F^{\circ M} = C^*_{env}(E) \otimes_{\min} C^*_{env}(F), \]
so $ \prop(E \otimes_{\min} F) \leq M$. Conversely, if we have that $m:=\prop(E\otimes_{\min} F) < M$, then either $E^{(m)} \subsetneq C^*_{env}(E)$ or $F^{(m)} \subsetneq C^*_{env}(F)$. So by Lemma \ref{lem:separation} this means that
\[ (E \otimes_{\min} F)^{\circ m} = E^{\circ m} \otimes_{\min} F^{\circ m} \neq C^*_{env}(E) \otimes_{\min} C^*_{env}(F), \]
but this is a contradiction because $m = \prop(E \otimes_{\min} F)$. So 
\[ M = \max\{ \prop(E),\prop(F) \} = \prop(E \otimes_{\min} F) \]
as claimed.
\end{proof}

\begin{remark}
Interestingly, while in this article we only work with unital operator systems, the equivalent result of Theorem \ref{thm:propagationNumber} for non-unital operator systems would directly prove invariance of the propagation number under stable equivalence. This was already proven by a different method in \cite{ConnesVSuijlekom2020}. Also interesting to mention in this regard are the recent results by Eleftherakis, Kakariadis and Todorov \cite{eleftherakisKakariadisTodorovMorita}, where they discuss Morita equivalence for operator systems, and show that two operator systems are equivalent in this way if and only if they are stably equivalent. So this implies that the propagation number is also an invariant under this type of Morita equivalence. 
\end{remark}

\subsection*{Acknowledgements}
I would like to sincerely thank Walter van Suijlekom, for the feedback in regards to this article, as well as for his support and guidance during the writing of the Master's thesis from which this article is adapted.

\bibliographystyle{ieeetr}

\bibliography{main.bib}

\end{document}